\documentclass[10pt,psamsfonts]{amsart}

\usepackage{amsthm}
\usepackage{amssymb}
\usepackage{amsmath}
\usepackage{graphicx}
\usepackage{amscd}
\usepackage{amsfonts}
\usepackage{amsbsy}
\usepackage[T1]{fontenc}
\usepackage[english]{babel}

\usepackage{epsfig}
\usepackage{amssymb}

\newtheorem{thm}{Theorem}[section]
\newtheorem{cor}{Corollary}[section]
\newtheorem{rmk}{Remark} [section]

\newtheorem{definition}{Definition}[section]
\newtheorem{lemma}{Lemma}[section]
\newtheorem{ex}{Example}[section]

\title[Expansive measures for flows]{Expansive measures for flows}

\author{D. Carrasco-Olivera, C. A. Morales}

\address{Instituto de Matem\'atica, Universidade Federal do Rio de Janeiro, P. O.
Box 68530, 21945-970 Rio de Janeiro, Brazil.}
\email{morales@impa.br.}
\address{Departamento de Matem\' atica, Universidad del B\' {i}o, B\' {i}o Av. Collao \# 1202,
Casilla 5-C, VIII-Regi\' on, Concepci\' on, Chile.}
\email{dcarrasc@zeus.dci.ubiobio.cl}

\thanks{C-O was partially supported by FONDECYT project 11121598, CONICYT (Chile) and  Post-Doctorado
Verano 2013, IMPA, Rio de Janeiro, Brazil. CAM was partially supported by CNPq, CAPES and PRONEX from Brazil.}

\subjclass[2010]{54H20, 34C35}

\keywords{Expansive Measure, Expansive Flow, Metric Space}

\begin{document}

\begin{abstract}
We extend the concept of expansive measure \cite{am} defined for homeomorphism to flows.
We obtain some properties for such measures including
abscense of singularities in the support,
aperiodicity, expansivity with respect to time-$T$ maps, invariance under flow-equivalence, negligibleness of orbits,
characterization of expansive measures for expansive flows
and naturallity under suspensions.
As an application we obtain a new proof of the well known fact that
there are no continuous expansive flows on surfaces (e.g. \cite{hs}).
\end{abstract}

\maketitle

\section{Introduction}

\noindent
{\em Expansive systems} have been important object of study by many authors since its introduction by Utz in the middle of the nineteen century \cite{u}.
In 1972's Bowen and Walters \cite{bw} proposed an extension of expansivity
to flow and proved that certain properties in the discrete case hold true in the flow's context too.
A variant of this definition named {\em unstable flows} was considered soon after by
Norton and O'Brien \cite{no}. Lewowicz \cite{lc} and Ruggiero \cite{r} considered slightly more general definitions
whereas Komuro \cite{k} introduced a kind of expansivity allowing non-isolated singularities (see also \cite{a}).
On the other hand, an extension of such a concept to Borel measures
was proposed by the second author on metric and uniform spaces,
in collaboration with Arbieto \cite{am} and Sirvent \cite{ms} respectively.
More precisely, they defined {\em expansive measures} for discrete systems
on such spaces and, in the metric case, they were able to use such measures to study ergodic systems with positive entropy.

In this paper we shall consider an extension of the notion of expansive measures to flows,
much in the spirit of the pioneering work by Bowen and Walters \cite{bw}.
Indeed, we define expansive measure for flows and prove that some of the properties
obtained by the second author and Arbieto \cite{am}, \cite{m} can be extended to the flow's context.
Such properties include
abscense of singularities in the support of expansive measures,
aperiodicity, expansivity with respect to time-$T$ maps, invariance under flow-equivalence, negligibleness of orbits,
characterization of expansive measures for expansive flows
and naturallity under suspensions.
As an application we obtain a probabilistic proof of the well known fact that
there are no continuous expansive flows on surfaces (e.g. \cite{hs}).
Our concept also represents a generalization of the notion of {\em pairwise sensitivity} \cite{cj} defined for maps to flows.

\section{Statement of the results}

\noindent
Hereafter $X$ will denote a metric space.
The closed and open ball operations will be denoted by $B[x,\delta]$ and $B(x,\delta)$
respectively, $\forall (x,\delta)\in X\times \mathbb{R}$.

A {\em flow} of $X$ is a map $\phi: \mathbb{R}\times X\to X$ satisfying $\phi(t,\phi(s,x))=\phi(t+s,x)$ for all $s,t\in \mathbb{R}$ and $x\in X$.
Given $I\subset \mathbb{R}$ and $A\subset X$ we set $\phi_I(A)=\bigcup_{t\in I}\phi_t(A)$ and write $\phi_I(x)$ instead of $\phi_I(\{x\})$.
In particular, $\phi_{\mathbb{R}}(x)$ is called the orbit of $x\in X$ under $\phi$.
Denote by $\mathcal{B}$ the set of continuous maps $h:\mathbb{R}\to \mathbb{R}$ with $h(0)=0$.

In the sequel we state the classical definition of expansive flow due to Bowen and Walters \cite{bw}.

\begin{definition}
\label{ex}
We say that a flow $\phi$ on $X$ is {\em expansive} if for every $\epsilon>0$
there is $\delta>0$ such that if $x,y\in X$ satisfy $d(\phi_t(x),\phi_{h(t)}(y))\leq\delta$ for every $t\in \mathbb{R}$ and some $h\in \mathcal{B}$,
then $y\in \phi_{(-\epsilon,\epsilon)}(x)$.
\end{definition}

Notice that in this definition there exists an expansivity constant $\delta$ but depending on $\epsilon$.
Ruggiero \cite{r} dismissed such an $\epsilon$-dependence in order to introduce the following definition.

\begin{definition}
\label{r-ex}
We say that a flow $\phi$ on $X$ is {\em Ruggiero expansive} if
there is $\delta>0$ such that if $x,y\in X$ satisfy $d(\phi_t(x),\phi_{h(t)}(y))\leq\delta$ for every $t\in \mathbb{R}$ and some $h\in \mathcal{B}$,
then $y\in \phi_{\mathbb{R}}(x)$.
\end{definition}

The original definition in \cite{r} required $h$ above to be surjective but this requirement is unnecessary
for continuous flows on compact metric spaces.
Taking $\epsilon=1$ in Definition \ref{ex} we get immediately that every expansive flow is Ruggiero expansive
(the converse is known for certain examples like the geodesic flows on closed manifolds without conjugated points).

We can reformulate Definition \ref{r-ex} by saying $\phi$ is Ruggiero expansive if and only if there is $\delta>0$ such that
\begin{equation}
\label{eq1}
\Gamma_\delta(x)\subset \phi_{\mathbb{R}}(x),
\quad\quad\forall x\in X,
\end{equation}
where
$$
\Gamma_\delta(x)=\displaystyle\bigcup_{h\in \mathcal{B}}\displaystyle\bigcap_{t\in\mathbb{R}}\phi_{-h(t)}(B[\phi_t(x),\delta]).
$$
It is this reformulation which allows us to define expansive measures for flows.
Recall that {\em Borel measure} is a non-negative $\sigma$-additive map $\mu$ defined in the Borel $\sigma$-algebra of $X$
(\cite{bo}).
For any subset $B\subset X$ we write $\mu(B)=0$ if $\mu(A)=0$ for every Borel set $A\subset B$.

\begin{definition}
 \label{exp-meas}
We say that a Borel measure $\mu$ of $X$ is {\em expansive} for a flow $\phi$ on $X$ if there is $\delta>0$ such that
$\mu(\Gamma_\delta(x))=0$, $\forall x\in X$.
Such a $\delta$ is called {\em expansivity constant} of $\mu$ (with respect to $\phi$).
\end{definition}

\begin{rmk}
Similar definition takes place for a topological group action
$\phi:X\times G\to X$. Indeed, we only have to replace $\mathbb{R}$ and $0$ by $G$ and the neutral element
of $G$ respectively. For uniform spaces replace the $\delta$ in the definition
of $\Gamma_\delta(x)$ by an entourage of the uniformity (compare with \cite{ms}).
\end{rmk}

Our first result gives a necessary and sufficient condition for
a Borel probability measure to be expansive.
To state it we say that a flow $\phi$ on $X$  is continuous if it does as a map
from $\mathbb{R}\times X$ into $X$, where $X\times \mathbb{R}$ is equipped with the product
metric.
If $x\in X$ satisfies $\phi_t(x)=x$ for every $t\in \mathbb{R}$,
then we say that $x$ is a {\em singularity} of $\phi$.

\begin{thm}
\label{A2}
Let $\phi$ be a continuous flow {\em without singularities}
of a compact metric space $X$.
Then, a Borel probability measure $\mu$ of $X$ is expansive for $\phi$
if and only if
there is $\alpha>0$ such that
$\mu(\Gamma_\alpha(x))=0$ for $\mu$-a.e. $x\in X$.
\end{thm}

To state our second result we will need more notations.
Let $\phi$ be a flow of $X$.
By a {\em periodic point} of $\phi$ we mean a point $x\in X$ for which there is a minimal
$t>0$ satisfying $\phi_t(x)=x$.
This minimal $t$ is the so-called {\em period} denoted by $t_x$.
Denote by $Sing(\phi)$ and $Per(\phi)$ the set of singularities and periodic points of $\phi$ respectively.

The {\em support} of a Borel measure $\mu$ on $X$ will be denoted by supp$(\mu)$.
Given another metric space $Y$ and a Borel measurable map $f: X\to Y$ we define
the pullback measure $f_*(\mu)=\mu\circ f^{-1}$ on $Y$. If $X=Y$ we say that $\mu$ is {\em invariant} for $f$ if $f_*\mu=\mu$.
For flows $\phi$ we say that $\mu$ is invariant for $\phi$ if it does for the time-$t$ map $\phi_t$, $t\in \mathbb{R}$.

We say that a Borel measure $\mu$ {\em vanishes along the orbits} of a flow $\phi$
if $\mu(\phi_\mathbb{R}(x))=0$ for every $x\in X$
(this is a non-trivial hypothesis, see \cite{b}).
It is clear that all such measures are {\em non-atomic}
(i.e. takes the value zero at single-point sets) but not conversely.
We also say that $\phi$ is {\em aperiodic} with respect to a Borel measure $\mu$
if $\mu(Per(\phi))=0$ (c.f. \cite{g}).
An {\em equivalence} between continuous flows $\phi$ on $X$ and
$\psi$ on another metric space $Y$ is a homeomorphism
$f: X\to Y$ carrying the orbits of $\phi$ onto orbits of $\psi$
(in such a case we say that the flows are {\em equivalent}).

With these definitions we can state our first result.

\begin{thm}
\label{general}
The following properties hold for every
continuous flow $\phi$ of a compact metric space and every Borel measure $\mu$:
\begin{enumerate}
\item[(A1)]
If $\mu$ is expansive for $\phi$, then
supp$(\mu)\cap Sing(\phi)=\emptyset$.
\item[(A2)]
If $\mu$ is expansive for
$\phi$, then $\phi$ is aperiodic with respect to $\mu$.
\item[(A3)]
If $f$ is an equivalence between $\phi$ and $\psi$, then $\mu$ 
is expansive for $\phi$ if and only if $f_*\mu$ is expansive for $\psi$.
\item[(A4)]
If $\mu$ is expansive for $\phi$, then $\mu$ vanishes along the orbits of $\phi$.
\item[(A5)]
If $\phi$ is expansive, then every Borel measure vanishing along the orbits of $\phi$
is expansive for $\phi$.
\end{enumerate}
\end{thm}

Let us put some comments in light of the above result.

First of all we can observe that (A1) represents a measurable version of the well
known property of expansive flows that their singularities consists of isolated points.
The aperiodicity property in (A2) has been proved to be an important tool in the ergodic
theory of flows (e.g. \cite{g}).
Notice that (A3) implies two things. Firstly, that the property of being an expansive measure independs on the metric and, secondly,
that if $\mu$ is expansive for $\phi$, then so does $(\phi_T)_*\mu$, for every $T\in \mathbb{R}$.
On the other hand, (A4) implies that every expansive measure is non-atomic, but,
unlike the homeomorphism case, there are non-atomic measures which are not
expansive for certain expansive flows (take for instance a measure supported on an orbit).
We will see below in Example \ref{el} that the converse of (A5) is false.

Next we recall the definition of expansive measure for homeomorphisms \cite{am}.

\begin{definition}
\label{ar-mor}
We say that a Borel measure $\mu$ of $X$ is {\em expansive}
for a homeomorphism $f: X\to X$ if there is $\delta>0$ such that
$\mu(\hat{\Gamma}_\delta(x))=0$ for every $x\in X$, where
$$
\hat{\Gamma}_\delta(x)=\{y\in X:d(f^n(x),f^n(y))\leq \delta, \forall n\in\mathbb{Z}\}.
$$
(Notation $\hat{\Gamma}_\delta^f(x)$ indicates dependence on $f$).
\end{definition}

Now, notice that for continuous flows $\phi$ it is easy to see that the time $T$-map $\phi_T: X\to X$
defined by $\phi_T(x)=\phi(T,x)$ is a homeomorphism of $X$, $\forall t\in \mathbb{R}$.
With this observation we state our second result.

\begin{thm}
\label{general2}
If $\mu$ is an expansive measure of a continuous flow
$\phi$ on a compact metric space,
then $\mu$ is expansive for $\phi_T$, $\forall T\neq 0$.
\end{thm}

The converse of this theorem is false by the following counterexample
which arose from discussions with Prof.  Arbieto:

\begin{ex}
It is well known that the so-called {\em geometric Lorenz attractor} \cite{abs}, \cite{gw}
supports an ergodic invariant probability measure with positive entropy (see for instance \cite{v}).
Such a measure was proved to be mixing \cite{lmp},
and so, ergodic for every time-$T$ map
with $T\neq0$.
From this an a result in \cite{am} we obtain that such a measure is also expansive for every time-$T$ maps $T\neq0$.
However, the support of such a measure has a singularity, and so, it is not expansive for the flow
by (A1) of Theorem \ref{general}.
\end{ex}

This example also shows that the result in \cite{am} that every ergodic invariant measure
with positive entropy of a homeomorphism of a compact metric space is expansive
is false for flows instead of homeomorphisms.
Furthermore, it motivates the question
if there are continuous flows on compact metric spaces
exhibiting a Borel measure, without equilibria in its support, which
is not expansive for the flow but expansive for every time-$T$ map with $T\neq0$.

To state our next result we will need the notion of suspension for homeomorphisms on metric spaces \cite{brw}, \cite{bw}.
Suppose that $X$ is compact. Then, we can assume that $diam(X)=1$
(otherwise we work with the equivalent metric $\frac{d}{diam(X)}$).
Given a continuous map $\tau: X\to (0,\infty)$ we define
$$
Y^\tau=\{(x,t)\in X\times \mathbb{R}:0\leq t\leq \tau(x)\},
$$
It is convenient to identify $X$ with the base $X\times \{0\}$ in $Y^\tau$.
If $f: X\to X$ is a homeomorphism we can
identify the points $(x,1)$ and $(f(x),0)$ of $Y^\tau$. This yields an equivalence relation whose class at
$(x,t)$ will still be denoted by $(x,t)$.
The corresponding quotient space will be denoted by $Y^{\tau,f}$.
The {\em suspension flow over $f$ with height function $\tau$} is the flow $\phi: Y^{\tau,f}\times \mathbb{R}\to Y^{\tau,f}$ defined
by $\phi^{\tau,f}_t(x,s)=(x,t+s)$.

There is a natural metric $d^{\tau,f}$ on $Y^{\tau,f}$ making it a compact metric space whenever $X$ is
(this is called {\em Bowen-Walter metric} in \cite{brw}). There is also
an operator $T^{\tau,f}: \mu\mapsto T^{\tau,f}(\mu)$ taking a Borel measure $\mu$ of $X$ into the Borel measure
$T^{\tau,f}(\mu)$ of $Y^{\tau,f}$ defined implicitely by
\begin{equation}
\label{eq9}
\int_{Y^{\tau,f}}h(y)dT^{\tau,f}(\mu)(y)=\int_X\frac{1}{\tau(x)}\int_0^{\tau(x)}h(\phi_t^{\tau,f}(x))dt\,d\mu(x)
\end{equation}
for every continuous map $h: Y^{\tau,f}\to \mathbb{R}$.
A nice property of this operator is that
a Borel measure $\mu$ of $X$ is a probability or non-atomic or invariant for $f$ depending on whether
there is a continuous map $\tau: X\to (0,\infty)$ such that
$T^{\tau,f}(\mu)$ is a probability or vanishes along the orbits of $\phi^{\tau,f}$ or invariant for the suspension
flow over $f$ with height function $\tau$ respectively.
Another interesting property is given below.

\begin{thm}
\label{general3}
If $f: X\to X$ is a homeomorphism of a compact metric space $X$,
then a Borel measure $\mu$ of $X$ is expansive for $f$
if and only if there is a continuous map $\tau: X\to (0,\infty)$ such that
$T^{\tau,f}(\mu)$ is expansive for suspension flow over $f$ with height function $\tau$.
\end{thm}

This result represents a link between expansive measures for homeomorphisms and their corresponding suspensions.
We can use it to produce a counterexample for the converse of
(A5) in Theorem \ref{general}.

\begin{ex}
\label{el}
Consider a compact metric space $X$ supporting
$n$-expansive homeomorphisms $f: X\to X$
which are not expansive \cite{mm} and
the suspension flow $\phi^{1,f}$ over $f$ with height function $\tau=1$.
It turns out that $Y^{1,f}$ has at least one measure vanishing along the orbits of $\phi^{1,f}$,
namely, $T^{1,f}(\mu)$ for some non-atomic measure $\mu$ on $X$.
Furthermore, every Borel measure vanishing along the orbits of $\phi^{1,f}$ is expansive
for $\phi^{1,f}$ but $\phi^{1,f}$ is not expansive since $f$ is not \cite{bw}.
\end{ex}

This counterexample motivates the study of the class of flows defined below.

\begin{definition} We say that a flow $\phi$ is {\em measure-expansive}
if every Borel measure vanishing along the orbits of $\phi$ is expansive for $\phi$.
\end{definition}

At first glance we can apply our results to prove the following result.
The term {\em closed surface} means a compact connected boundaryless manifold
of dimension two.
By {\em non-trivial recurrence} of a flow $\phi_t$ on a metric space $X$ we mean a non-periodic
point $x_0$ which is {\em recurrent} in the sense that $x_0\in \omega(x_0)$, where
$$
\omega(x)=\left\{y\in X:y=\lim_{n\to \infty}\phi(x,t_n)\mbox{ for some sequence }t_n\to\infty\right\},
\quad\quad\forall x\in X.
$$

\begin{thm}
\label{chino}
There are no continuous measure-expansive flows with non-trivial recurrence of closed surfaces.
\end{thm}

\begin{rmk}
\label{rrr}
We stress that the non-trivial recurrence hypothesis above is unnecessary and helps to simplify the proof.
\end{rmk}

Theorem \ref{general}-(A5) implies that every continuous expansive flow of a compact metric space is measure-expansive.
Therefore,
Theorem \ref{chino} and Remark \ref{rrr} imply the following result from \cite{hs}.

\begin{cor}
There are no continuous expansive flows of closed surfaces.
\end{cor}

\section{Preliminars}

\noindent
We start with the following lemma.

\begin{lemma}
\label{l1}
Every expansive measure of a continuous flow $\phi$ on a compact metric space vanishes along the orbits of $\phi$.
\end{lemma}

\begin{proof}
Let $\mu$ be an expansive measure of $\phi$ with expansivity constant $\delta$.
It follows from compactness that there is
$\alpha>0$ such that $\phi_{(-\alpha,\alpha)}(x)\subset \Gamma_\delta(x)$ (and so $\mu(\phi_{(-\alpha,\alpha)}(x))=0$), $\forall x\in X$
(c.f. p. 506 in \cite{a}).

Given $x\in X$ we have two possibilities:
either $\phi_\mathbb{R}(x)$ is a closed (and hence Lindel\"of) subset of $X$
or the induced map $\phi^x: \mathbb{R}\to \phi_{\mathbb{R}}(x)$ defined by $\phi^x(t)=\phi(x,t)$ is a homeomorphism onto $\phi_\mathbb{R}(x)$.
In any case we can arrange sequences $x_k\in X$ and $\delta_k\in (0,\alpha)$ such that
$\{\phi_{(-\delta_k,\delta_k)}(x_k):k\in \mathbb{N}\}$ covers $\phi_{\mathbb{R}}(x)$ so
$$
\mu(\phi_{\mathbb{R}}(x))\leq \displaystyle\sum_{k\in\mathbb{N}}\mu(\phi_{(-\delta_k,\delta_k)}(x_k))=0.
$$
As $x\in X$ is arbitrary we conclude that $\mu(\phi_\mathbb{R}(x))=0$, $\forall x\in X$, thus proving the result.
\end{proof}

Another property of the expansive measures is given below.

\begin{lemma}
\label{sing-supp}
If $\phi$ is a flow of a compact metric space $X$, then $Sing(\phi)\cap supp(\mu)=\emptyset$ for every expansive measure $\mu$ of $\phi$.
\end{lemma}

\begin{proof}
If $\sigma\in Sing(\phi)$ we can prove as in Lemma 1 of \cite{bw} that
$B(\sigma,\delta)\subset \Gamma_\delta(\sigma)$, $\forall\delta>0$.
Taking $\delta$ as an expansivity constant we obtain $\mu(B(\sigma,\delta))=0$
so $\sigma\not\in supp(\mu)$.
\end{proof}

We shall use the following lemma which is essentially contained in \cite{bw}.
We include its proof for the sake of completeness.

\begin{lemma}
 \label{thomas1}
Let $\phi$ be a continuous flow without singularities of a compact metric space $X$.
Then, for every $\alpha>0$ there is $\delta>0$ such that for every $x,y\in X$ satisfying $y\in \Gamma_\delta(x)$ there is an increasing homeomorphism $\hat{h}\in \mathcal{B}$ satisfying
$d(\phi_t(x),\phi_{\hat{h}(t)}(y))<\alpha$, $\forall t\in\mathbb{R}$.
\end{lemma}

\begin{proof}
As in the proof of (ii) $\Rightarrow$ (i) in Theorem 3 of \cite{bw} we have that there is $T_0>0$ such that for every $0<T<T_0$ there are $\delta_T>0$ and $\tau_T>0$ such that
if $h\in \mathcal{B}$ satisfies $d(\phi_t(x),\phi_{h(t)}(y))<\delta_T$ for every $t\in\mathbb{R}$ and some $x,y\in X$,
then $h(t+T)-h(t)\geq\tau_T$ for every $t\in\mathbb{R}$.

Fix $\alpha>0$ and take $0<T<\frac{T_0}{3}$ such that
$$
\displaystyle\sup_{(z,u)\in X\times [0,T]}d(z,\phi_u(z))\leq \frac{\alpha}{2}.
$$
For this $T$ we fix $\delta_T$ and $\tau_T$ as above and also
$0<\delta<\min\left(\delta_T,\frac{\alpha}{2}\right)$.

Now suppose that $x,y\in X$ satisfies $y\in \Gamma_\delta(x)$.
Then, there is $h\in \mathcal{B}$ such that
$d(\phi_t(x),\phi_{h(t)}(y))<\delta$ for every $t\in\mathbb{R}$. As $\delta<\delta_T$ we have
$d(\phi_t(x),\phi_{h(t)}(y))<\delta_T$ for every $t\in\mathbb{R}$, and so, by the property of $\delta_T$, we obtain
$h(t+T)-h(t)\geq\tau_T$ for every $t\in\mathbb{R}$.
Define $\hat{h}: \mathbb{R}\to \mathbb{R}$ by $\hat{h}(nT)=h(nT)$ ($\forall n\in\mathbb{Z}$) and extend by linearity to $[nT,(n+1)T]$ for every $n\in\mathbb{Z}$.
Clearly $\hat{h}(0)=0$ and using the property of $\tau_T$ we see that $\hat{h}$ is an increasing homeomorphism.
Moreover, for every $t\in [nT,(n+1)T]$ (for some integer $n$) there is $t'\in [nT,(n+1)T]$ with $\hat{h}(t)=h(t')$ and so
$$
d(\phi_t(x),\phi_{\hat{h}(t)}(y))=d(\phi_t(x),\phi_{h(t')}(y))
\leq 
d(\phi_t(x),\phi_{t'}(x))+d(\phi_{t'}(x),\phi_{h(t')}(y))\leq
$$
$$
\left(\displaystyle\sup_{(z,u)\in X\times [0,T]}d(z,\phi_u(z))\right)+\delta<\frac{\alpha}{2}+\frac{\alpha}{2}=\alpha,
$$
proving $d(\phi_t(x),\phi_{\hat{h}(t)}(y))<\alpha$ for all $t\in\mathbb{R}$.
\end{proof}

\begin{cor}
\label{c1}
If $\phi$ is a continuous flow without singularities of a compact metric space, then for every
$\alpha>0$
there is $\delta>0$ such that
$\Gamma_\delta(x)\subset \Gamma_\alpha(y)$ for every $x,y\in X$ with $y\in \Gamma_\delta(x)$.
\end{cor}

\begin{proof}
Fix $\alpha>0$ and let $\delta$ be as in Lemma \ref{thomas1} for $\frac{\alpha}{2}$.
By Lemma \ref{thomas1} if $y\in \Gamma_\delta(x)$ and $z\in \Gamma_\delta(x)$ there are increasing homeomorphisms
$\hat{h},\tilde{h}\in \mathcal{B}$ such that
$d(\phi_t(x),\phi_{\hat{h}(t)}(y))\leq \frac{\alpha}{2}$ and $d(\phi_t(x),\phi_{\tilde{h}(t)}(z))\leq\frac{\alpha}{2}$
for every $t\in \mathbb{R}$.
It follows that
$d(\phi_{\hat{h}(t)}(y),\phi_{\tilde{h}(t)}(z))\leq\alpha$ for all $t\in \mathbb{R}$.
Replacing $t$ by $\hat{h}^{-1}(t)$ in this relation we obtain
$$
d\left(\phi_t(y),\phi_{(\tilde{h}\circ \hat{h}^{-1})(t))}(z)\right)\leq\alpha,
\quad\quad\forall t\in \mathbb{R}.
$$
As $\tilde{h}\circ \hat{h}^{-1}\in\mathcal{B}$ we conclude that $z\in \Gamma_\alpha(y)$ and the proof follows.
\end{proof}

Combining the proof of theorems 3 and 5 in \cite{bw} we obtain the following lemma.
In its proof we denote by $[r]$ the integer part of $r\in \mathbb{R}$.

\begin{lemma}
\label{period}
Let $\phi$ be a continuous flow on a compact metric space $X$.
Then, for every $\delta>0$ and $t>0$ there are $0<\alpha_0<\min\{\delta,2t\}$ and $0<\beta_0<\min\{\delta,2t\}$ such that
if $x,y\in X$ satisfy $\phi_a(x)=x$, $\phi_b(y)=y$ for some
$a,b\in \left[t-\frac{\alpha}{2},t+\frac{\alpha}{2}\right]$ and $d(x,y)\leq \beta$ with $0<\alpha<\alpha_0$ and $0<\beta<\beta_0$,
then $y\in \Gamma_\delta(x)$.
\end{lemma}

\begin{proof}
Fix $\delta>0$ and $t>0$.

Take $0<\beta_0<\min\{\delta,2t\}$ such that
if $z,w\in X$ satisfy
$d(z,w)\leq \beta$ for some $0<\beta<\beta_0$, then
\begin{equation}
\label{beta}
d(\phi_s(z),\phi_s(w))\leq \frac{\delta}{2},
\quad\quad\forall 0\leq s\leq t+1.
\end{equation}
Take $0<\alpha_0<\min\{\delta,2t\}$ such that if $0<\alpha<\alpha_0$ then 
\begin{equation}
\label{ll}
\max\left\{\alpha,2\left(\displaystyle\sup_{(z,u)\in X\times [0,\alpha]}d(\phi_u(z),z)\right)\right\}<\min\left\{\frac{\delta}{2},1\right\}.
\end{equation}

Fix $x,y\in X$ satisfying $\phi_a(x)=x$, $\phi_b(y)=y$ for some
$a,b\in \left[t-\frac{\alpha}{2},t+\frac{\alpha}{2}\right]$ and $d(x,y)\leq \beta$ with $0<\alpha<\alpha_0$ and $0<\beta<\beta_0$.

Put $m=\left[\frac{t-\frac{\alpha}{2}}{\alpha}\right]+1$ and define
$$
t_{pm+q}=pa+q\alpha\quad\mbox{ and }\quad u_{pm+q}=pb+q\alpha,
\quad\forall (p,q)\in \mathbb{Z}\times \{0,\cdots, m-1\}.
$$
Direct computations show $t_{pm+q+1}-t_{pm+q}=\alpha$ (for $(p,q)\in \mathbb{Z}\times \{0,\cdots, m-2\}$),
$t_{(p+1)m}-t_{pm+m-1}\geq0$ (for $p\in\mathbb{Z}$) and analogously replacing $t$ by $u$. Moreover,
\begin{equation}
\label{eq2}
|t_{pm+q+1}-t_{pm+q}|\leq\alpha \quad\mbox{ and }\quad
|u_{pm+q+1}-u_{pm+q}|\leq\alpha
\end{equation}
for every $(p,q)\in \mathbb{Z}\times \{0,\cdots, m-1\}$.
Also
$$
d(\phi_{t_{pm+q}}(x),\phi_{u_{pm+q}}(y))=d(\phi_{q\alpha}(x),\phi_{q\alpha}(y))
$$
and
$$
0\leq q\alpha<m\alpha=\left[\frac{t-\frac{\alpha}{2}}{\alpha}\right]\alpha+\alpha
=t-\frac{\alpha}{2}-r+\alpha
$$
(for some $0\leq r<\alpha$) so
$$
0\leq q\alpha<t+\frac{\alpha}{2}\leq t+1, \quad\quad\forall q=1,\cdots, m-1.
$$
Since $d(x,y)\leq \beta$ we can put $s=q\alpha$ in (\ref{beta}) to obtain $d(\phi_{q\alpha}(x),\phi_{q\alpha}(y))\leq\frac{\delta}{2}$ so
\begin{equation}
\label{eq3}
d(\phi_{t_{pm+q}}(x),\phi_{u_{pm+q}}(y))\leq\frac{\delta}{2},
\quad\quad\forall (p,q)\in \mathbb{Z}\times \{0,\cdots, m-1\}.
\end{equation}
Now define $h: \mathbb{R}\to \mathbb{R}$ by setting $h(t_{pm+q})=u_{pm+q}$ and extending linearly
to $[t_{pm+q},t_{pm+q+1}]$, $\forall (p,q)\in \mathbb{Z}\times \{0,\cdots, m-1\}$. Clearly $h\in\mathcal{B}$. Moreover,
if $s\in \mathbb{R}$ then $s\in [t_{pm+q},t_{pm+q+1}[$ for unique $(p,q)\in \mathbb{Z}\times\{0,\cdots, m-1\}$, so
$$
d(\phi_s(x),\phi_{h(s)})\leq d(\phi_s(x),\phi_{t_{pm+q}}(x))+
d(\phi_{t_{pm+q}}(x),\phi_{u_{pm+q}}(y))+
$$
$$
d(\phi_{u_{pm+q}}(y),\phi_{h(s)}(y))
\overset{(\ref{eq2})}{\leq} 
2\left(\displaystyle\sup_{(z,u)\in X\times [0,\alpha]}d(\phi_u(z),z)\right)+d(\phi_{t_{pm+q}}(x),\phi_{u_{pm+q}}(y))
\overset{(\ref{ll}),(\ref{eq3})}{\displaystyle\leq}
$$
$$
\frac{\delta}{2}+\frac{\delta}{2}=\delta.
$$
We conclude that there is $h\in \mathcal{B}$ satisfying
$$
d(\phi_s(x),\phi_{h(s)}(y))<\delta,
\quad\quad\forall s\in \mathbb{R},
$$
so $y\in \Gamma_\delta(x)$.
\end{proof}

We also need a lemma about the localization of the periods of the periodic points.
Precisely, for every flow $\phi$ and every $x\in Per(\phi)$ we denote by $t_x=\min\{t>0:\phi_t(x)=x\}$.
Given $I\subset \mathbb{R}$ we denote by $Per_I(\phi)=\{x\in Per(\phi):t_x\in I\}$.
The special case for $I=[0,T]$, $T\geq 0$, will be denoted by $Per_T(\phi)$.
In particular, $Sing(\phi)=Per_0(\phi)$.

\begin{cor}
\label{c2}
Let $\phi$ be a continuous flow on a compact metric space $X$.
Then, for every $\delta>0$, $N\in \mathbb{N}^+$ and $x\in Per(\phi)\setminus Sing(\phi)$
there are $0<\alpha_*<\min\{\delta,2t_x\}$ and $0<\beta_*<\min\{\delta,2t_x\}$ such that
$$
B[x,\beta]\cap
\left(
\displaystyle\bigcup_{k=1}^NPer_{[kt_x-\frac{\alpha}{2},kt_x+\frac{\alpha}{2}]}(\phi)
\right)
\subset\Gamma_\delta(x),
\quad\quad\forall 0<\alpha<\alpha_*, \forall 0<\beta<\beta_*.
$$
\end{cor}

\begin{proof}
Fix $\delta>0$, $N\in\mathbb{N}^+$ and $x\in Per(\phi)\setminus Sing(\phi)$ (thus $t_x>0$).
Given $k=1,\cdots, N$ we put $t=kt_x$ in Lemma \ref{period} to obtain the corresponding numbers
$0<\alpha_k<\min\{\delta,2kt_x\}$ and $0<\beta_k<\min\{\delta,2kt_x\}$.
Choose
$$
\alpha_*=\min_{1\leq k\leq N}\alpha_k \quad\mbox{ and }\quad \beta_*=\min_{1\leq k\leq N}\beta_k.
$$
Clearly $0<\alpha_*<\min\{\delta,2t_x\}$ and $0<\beta_*<\min\{\delta,2t_x\}$.
Now fix $0<\alpha<\alpha_*$ and $0<\beta<\beta_*$. Take
$y\in B[x,\beta]\cap Per_{[kt_x-\frac{\alpha}{2},kt_x+\frac{\alpha}{2}]}(\phi)$ for some $k=1,\cdots, N$.
Putting $a=kt_x$ and $b=t_y$ we get $\phi_a(x)=x$ and $\phi_b(y)=y$ for some
$a,b\in [kt_x-\frac{\alpha}{2},kt_x+\frac{\alpha}{2}]$.
Since $0<\alpha<\alpha_*\leq \alpha_k$ and $0<\beta<\beta_*\leq \beta_k$ we can put $t=kt_x$ in Lemma \ref{period} to obtain
$y\in \Gamma_\delta(x)$. This ends the proof.
\end{proof}

Another localization lemma for periodic points is given below.

\begin{lemma}
\label{localiza}
Let $\phi$ be a continuous flow of a compact metric space.
Then, for every $T>0$, $x\in Per_T(\phi)\setminus Sing(\phi)$ and $0<\alpha<t_x$ there is $\beta_0>0$ such that
$$
B[x,\beta]\cap Per_T(\phi)\subset \displaystyle\bigcup_{k=1}^{\left[\frac{T}{t_x}\right]}Per_{[kt_x-\frac{\alpha}{2},kt_x+\frac{\alpha}{2}]}(\phi),
\quad\quad\forall 0<\beta<\beta_0.
$$
\end{lemma}

\begin{proof}
Otherwise there are $T>0$, $x\in Per_T(\phi)\setminus Sing(\phi)$ and $0<\alpha<t_x$ such that
\begin{equation}
\label{eq5}
x_n\not\in
\displaystyle\bigcup_{k=1}^{\left[\frac{T}{t_x}\right]}Per_{[kt_x-\alpha,kt_x+\alpha]}(\phi),
\end{equation}
for some sequence $x_n\in Per_T(\phi)$ with $x_n\to x$. Clearly $t_{x_n}\in [0,T]$ for all $n$ so we can assume that
$t_{x_n}\to t_\infty$ for some $t_\infty\in [0,T]$. Since $x\notin Sing(\phi)$ and $x_n\to x$ we have $t_\infty>0$.
Moreover, by noting that $\phi_{t_{x_n}}(x_n)=x_n\to x$ and $\phi_{t_{x_n}}(x_n)\to \phi_{t_\infty}(x)$ we obtain
$\phi_{t_\infty}(x)=x$ thus $t_\infty=kt_x$ for some $1\leq k\leq \left[\frac{T}{t_x}\right]$.
It follows that $x_n\in Per_{[kt_x-\alpha,kt_x+\alpha]}(\phi)$ for $n$ large contradicting (\ref{eq5}).
\end{proof}

We shall apply the last two results to prove the following.

\begin{lemma}
\label{p1}
If $\mu$ is an expansive measure of a continuous flow $\phi$ on a compact metric space $X$,
then $\mu(Per_T(\phi))=0$ for every $T\geq 0$.
\end{lemma}

\begin{proof}
First notice that since $\phi$ is continuous and $X$ compact we have that $Per_T(\phi)$ is compact (and so a Borelian) for all $T\geq 0$.
Now, suppose by contradiction that there is $T\geq0$ satisfying
$\mu(Per_T(\phi))>0$.
Since $Per_T(\phi)$ is compact we can arrange $x_*\in Per_T(\phi)$ and $\beta_->0$ such that
\begin{equation}
\label{eq6}
\mu(B[x_*,\beta]\cap Per_T(\phi))>0,
\quad\quad\forall 0<\beta<\beta_-.
\end{equation}
In particular, $x_*\in supp(\mu)$ so $x_*\not\in Sing(\phi)$ by Lemma \ref{sing-supp} thus $T>0$.

Now fix and expansivity constant $\delta$ of $\mu$ and take $N=\left[\frac{T}{t_{x_*}}\right]$.
Let $\alpha_*$ and $\beta_*$ be the corresponding numbers for $x=x_*$ in Corollary \ref{c2}.
We also fix $0<\alpha<\alpha_*$ which together with $x=x_*$ and $T$ yields the constant $\beta_0$ as in Lemma \ref{localiza}.
Fix
$$
0<\beta<\min\{\beta_0,\beta_-,\beta_*\}.
$$
On the one hand,
applying Lemma \ref{localiza} we obtain
$$
B[x_*,\beta]\cap Per_T(\phi)\subset \displaystyle\bigcup_{k=1}^{\left[\frac{T}{t_x}\right]}Per_{[kt_{x_*}-\frac{\alpha}{2},kt_{x_*}+\frac{\alpha}{2}]}(\phi),
$$
so
$$
B[x_*,\beta]\cap Per_T(\phi)\subset B[x_*,\beta]\cap\left(\displaystyle\bigcup_{k=1}^{\left[\frac{T}{t_x}\right]}Per_{[kt_{x_*}-\frac{\alpha}{2},kt_{x_*}+\frac{\alpha}{2}]}(\phi)\right).
$$
On the other hand, since $0<\alpha<\alpha_*$ we can take $N=\left[\frac{T}{t_{x_*}}\right]$ in Corollary \ref{c2} to obtain
$$
B[x_*,\beta]\cap
\left(
\displaystyle\bigcup_{k=1}^{\left[\frac{T}{t_x}\right]}Per_{[kt_{x_*}-\frac{\alpha}{2},kt_{x_*}+\frac{\alpha}{2}]}(\phi)
\right)
\subset\Gamma_\delta(x_*).
$$
Putting all this together we obtain
$$
\mu(B[x_*,\beta]\cap Per_T(\phi))\leq \mu\left(B[x_*,\beta]\cap
\left(\displaystyle\bigcup_{k=1}^{\left[\frac{T}{t_{x_*}}\right]}Per_{[kt_{x_*}-\frac{\alpha}{2},kt_{x_*}+\frac{\alpha}{2}]}(\phi)
\right)\right)\leq
$$
$$
\mu(\Gamma_\delta(x_*))=0.
$$
Since we have chosen $0<\beta<\beta_-$ we obtain a contradiction by (\ref{eq6}).
\end{proof}

The following result establishes a relation between the expansivity of measures for flows $\phi$
and for its corresponding time $T$-maps $\phi_T$.

\begin{lemma}
\label{lele}
Let $\phi$ be a continuous flow of a compact metric space $X$.
Then, for every $\delta>0$ and $T\in \mathbb{R}\setminus\{0\}$ there is $\alpha>0$ such that
$\hat{\Gamma}_\alpha^{\Phi_T}(x)\subset \Gamma_\delta(x)$ for every $x\in X$.
\end{lemma}

\begin{proof}
Take $\alpha>0$ satisfying
\begin{equation}
\label{leco}
z,w\in X\mbox{ and }d(z,w)\leq \alpha\mbox{ implies }d(\phi_s(z),\phi_s(w))\leq \delta,
\quad\forall s\in [0,T].
\end{equation}
Take also $x,y\in X$ with $y\in \hat{\Gamma}_\alpha^{\phi_{nT}}(x)$, i.e.,
\begin{equation}
\label{lecoo}
d(\phi_{nT}(x),\phi_{nT}(y))\leq\alpha,
\quad\quad\forall n\in \mathbb{Z}.
\end{equation}
Since $T\neq0$
every $t\in \mathbb{R}$ satisfies $nT\leq t<(n+1)T$ for a unique $n\in \mathbb{Z}$ therefore
$$
d(\phi_t(x),\phi_t(y))=d(\phi_{t-nT}(\phi_{nT}(x)),\phi_{t-nT}(\phi_{nT}(y))
=d(\phi_s(z),\phi_s(w))\leq \delta
$$
by applying (\ref{leco}) and (\ref{lecoo}) to $z=\phi_{nT}(x)$, $w=\phi_{nT}(y)$ and $s=t-nT\in [0,T]$.

We conclude that there is
$h\in \mathcal{B}$ (i.e. the identity) satisfying
$$
d(\phi_t(x),\phi_{h(t)}(y))\leq \delta,
\quad\quad\forall t\in \mathbb{R},
$$
so $y\in \Gamma_\delta(x)$.
\end{proof}

Next we study the interplay between flow-equivalence and expansive measure for flows.
Clearly if $f: X\to Y$ is an equivalence between the flows $\phi$ and $\psi$ on $X$ and $Y$ respectively, then
$$
f(Sing(\phi))=Sing(\psi).
$$
Furthermore, if $\hat{f}(\psi)$ is the flow on $X$ defined by
$$
\hat{f}(\psi)(x,t)=f^{-1}(\psi(f(x),t)),
$$
then for every $x\in X$ there is a homeomorphism
$h_x\in \mathcal{B}$ satisfying
\begin{equation}
\label{paja}
\hat{f}(\psi)_t(x)=\phi_{h_x(t)}(x), \quad\quad\forall t\in \mathbb{R}.
\end{equation}
Indeed, for $x\not\in Sing(\phi)$ the existence of such an $h_x$ was proved in \cite{bw}
whereas for $x\in Sing(\phi)$ we can take $h_x(t)=t$.
We shall use this in the following lemma.

\begin{lemma}
\label{mane}
Let $f: X\to Y$ be an equivalence between continuous flows $\phi$ on $X$ and $\psi$ on $Y$, where $X$ and $Y$ are compact metric spaces.
Then, for every $\delta>0$ there is $\alpha>0$ such that
$$
f^{-1}(\Gamma_{\alpha,\psi}(z))\subset \Gamma_{\delta,\phi}(f^{-1}(z)),
\quad\quad\forall z\in Y.
$$
\end{lemma}

\begin{proof}
Fix $\delta>0$.
By compactness we have that $f^{-1}$ is uniformly continuous, so,
there is $\alpha>0$ such that
$d(f^{-1}(z), f^{-1}(w))\leq\delta$ whenever $z,w\in Y$ satisfy $d(z,w)\leq\alpha$.

Now, take $z,w\in Y$ with $w\in \Gamma_{\alpha,\psi}(z)$, i.e., there is $h\in \mathcal{B}$ such that
$$
d(\psi_t(z),\psi_{h(t)}(w))\leq\alpha,
\quad\quad\forall t\in \mathbb{R}.
$$
Therefore, the choice of $\alpha$ implies
$$
d(f^{-1}(\psi_t(z)),f^{-1}(\psi_{h(t)}(w)))\leq\delta,
\quad\quad\forall t\in \mathbb{R},
$$
so,
$$
d(\hat{f}(\psi)_t(f^{-1}(z)),\hat{f}(\psi)_{h(t)}(f^{-1}(w)))\leq\delta,
\quad\quad\forall t\in \mathbb{R},
$$
and then
$$
d(\phi_{h_{f^{-1}(z)}(t)}(f^{-1}(z)),\phi_{(h_{f^{-1}(w)}\circ h)(t)}(f^{-1}(w)))\leq\delta,
\quad\quad\forall t\in \mathbb{R}
$$
by (\ref{paja}).
Since $h_{f^{-1}(z)}$ is a homeomorphism we can replace $t$ by $h_{f^{-1}(z)}^{-1}(t)$ in the expression above to obtain
$$
d(\phi_t(f^{-1}(z)),\phi_{\hat{h}(t)}(f^{-1}(w)))\leq\delta,
\quad\quad\forall t\in \mathbb{R},
$$
where $\hat{h}=h_{f^{-1}(w)}\circ h\circ h_{f^{-1}(z)}^{-1}$. As clearly $\hat{h}\in\mathcal{B}$ we conclude that
$f^{-1}(w)\in \Gamma_{\delta,\phi}(f^{-1}(z))$. As $z\in Y$ and $w\in \Gamma_{\alpha,\psi}(z)$ are arbitrary we conclude that
$$
f^{-1}(\Gamma_{\alpha,\psi}(z))\subset \Gamma_{\delta,\phi}(f^{-1}(z)), \quad\quad\forall z\in Y.
$$
\end{proof}

The lemma below is contained in the proof of Theorem 6 of \cite{bw}.
We include its proof here for the sake of completeness.

\begin{lemma}
\label{suspension1}
If $f: X\to X$ is a homeomorphism of a compact metric space $X$, then
$$
\Gamma_{\delta,\phi^{1,f}}(y)\subset \hat{\Gamma}_\delta^f(x)\times [0,1],
\quad\quad\forall y=(x,t)\in Y^{1,f}\mbox{ and }0<\delta<\frac{1}{4},
$$
where $\hat{\Gamma}^f_\delta(x)$ above denotes the dynamic ball in Definition \ref{ar-mor}
with respect to the metric $d'(x_1,x_2)=\min\{d(x_1,x_2),d(f(x_1),f(x_2))\} $ on $X$.
\end{lemma}

\begin{proof}
Fix $0<\delta<\frac{1}{4}$, $y=(x,t)$ and $y_0=(x_0,t_0)$ in $Y^{1,f}$
with $y_0\in \Gamma_{\delta,\phi^{1,f}}(y)$.
Then, there is $h\in \mathcal{B}$ such that
$d^{1,f}(\phi^{1,f}_t(y),\phi^{1,f}_{h(t)}(y_0))\leq \delta$ for every $t\in \mathbb{R}$.

We divide the proof that $y_0\in \hat{\Gamma}_\delta(x)\times [0,1]$
in two cases:

\noindent
{\em Case 1: $t=\frac{1}{2}$}.

As $\phi^{1,f}_1(y)=(x,\frac{3}{2})$, $\phi^{1,f}_{h(1)}(y_0)=(x_0,t_0+h(1))$ and
$d^{1,f}(\phi_1^{1,f}(y),\phi_{h(1)}^{1,f}(y_0))\leq\delta$ one has
$$
\left|
\frac{3}{2}-t_0-h(1)\right|\leq \delta.
$$
As $0<\delta<\frac{1}{4}$ we get
$0\leq t_0+h(1)\leq 1$ thus
$\phi^{1,f}_{h(1)}(y_0)=(f(x_0),t_1)$ for some $0\leq t_1\leq 1$.
As $\phi_1^{1,f}(y)=\left(f(x),\frac{1}{2}\right)$ we obtain
$d'(f(x),f(x_0))\leq d^{1,f}(\phi_1^{1,f}(y),\phi^{1,f}(y_0))\leq\delta$
yielding
$d'(f(x),f(x_0))\leq\delta.$
Analogously we obtain $d'(f^n(x),f^n(x_0))\leq\delta$ for all $n\in\mathbb{Z}$ thus
$x_0\in \hat{\Gamma}_\delta^f(x)$.
As $0\leq t_0\leq 1$ by the definition of $Y^{1,f}$ we obtain
$y_0\in \hat{\Gamma}_\delta(x)\times [0,1]$.

\vspace{5pt}

\noindent
{\em Case 2: $t\neq\frac{1}{2}$}.

Clearly there is a number $r$ with $|r|\leq\frac{1}{2}$ such that
$\phi_r^{1,f}(y)=\left(x,\frac{1}{2}\right)$.
Setting $y'=\phi_r^{1,f}(y)$ and $y_0'=\phi_{h(r)}^{1,f}(y_0)$
we have
$$
d^{1,f}(\phi_t^{1,f}(y'),\phi_{\hat{h}(t)}(y_0'))\leq\delta,
\quad\quad\forall t\in \mathbb{R},
$$
where
$\hat{h}(t)=h(t+r)-h(r)$.
Moreover,
$y_0'=(x_0,t_0+h(r))$ so
$\left|\frac{1}{2}-t_0-h(r)\right|\leq d^{1,f}(y',y_0')\leq\delta<\frac{1}{4}$ thus
$\frac{1}{4}\leq t_0+h(r)\leq \frac{3}{4}$.
Since $\hat{h}\in\mathcal{B}$ and
$y'=\left(x,\frac{1}{2}\right)$, Case 1 yields
$x_0\in \hat{\Gamma}_\delta^f(x)$ so
$y_0\in \hat{\Gamma}_\delta(x)\times [0,1]$.
\end{proof}

\section{Proofs}

\begin{proof}[Proof of Theorem \ref{A2}]
The necessity is clear so we only need to prove the sufficiency.
For this, we let $\alpha>0$ as in the statement and
$\delta$ be as in Corollary \ref{c1} for this $\alpha$.
Suppose by contradiction  that $\mu$ is not expansive.
Then, there is $x_0\in X$ and a Borel set $B\subset \Gamma_\delta(x_0)$ such that $\mu(B)>0$.
Since $\mu(B)>0$ we have that there is $y_0\in B$ such that
$\mu(\Gamma_\alpha(y_0))=0$.
But $y_0\in B\subset \Gamma_\delta(x_0)$ so $y_0\in \Gamma_\delta(x_0)$ thus
$\Gamma_\delta(x_0)\subset \Gamma_\alpha(y_0)$ by Corollary \ref{c1}.
As $B\subset \Gamma_\delta(x_0)$ we obtain $B\subset \Gamma_\alpha(y_0)$
so $\mu(B)=0$ since $\mu(\Gamma_\alpha(y_0))=0$, a contradiction which proves the result.
\end{proof}

\begin{proof}[Proof of Theorem \ref{general}]
(A1) follows directly from Lemma \ref{sing-supp}.

To prove (A2) let $\phi$ be a continuous flow of a compact metric space.
Clearly
$Per(\phi)=\bigcup_{N\in\mathbb{N}} Per_N(\phi)$. Since
$\mu(Per_N(\phi))=0$ for every $N\in\mathbb{N}$ and every expansive measure $\mu$
by Lemma \ref{p1}, we obtain
the result.

To prove (A3) we only need to prove the necessity as the sufficiency follows analogously
replacing $f$ by $f^{-1}$.
Let $\mu$ an expansive measure of $\phi$ and $\alpha$ be as in Lemma \ref{mane} for an expansivity constant $\delta$ of $\mu$.
Fix $z\in Y$ and a Borel set $B\subset \Gamma_{\alpha,\psi}(z)$.
Applying Lemma \ref{mane} we obtain $f^{-1}(B)\subset \Gamma_{\delta,\phi}(f^{-1}(z))$ so $\mu(f^{-1}(B))=0$ since
$\mu(\Gamma_{\delta,\phi}(f^{-1}(z)))=0$.
It follows that $f_*\mu(B)=\mu(f^{-1}(B))=0$ whence $f_*\mu(\Gamma_{\alpha,\psi}(z))=0$, $\forall z\in Y$. We conclude that
$f_*\mu$ is expansive for $\psi$ with expansivity constant $\alpha$.

Finally, (A4) is a consequence of Lemma \ref{l1} whereas (A5) follows from
the characterization of expansive flows (\ref{eq1}).
\end{proof}

\begin{proof}[Proof of Theorem \ref{general2}]
Fix $T\in \mathbb{R}\setminus\{0\}$ and let $\mu$ be an expansive measure of $\phi$ with expansivity constant $\delta$.
Taking $\alpha$ as in Lemma \ref{lele} for such $\delta$ and $T$
we obtain
$\mu(\hat{\Gamma}_\alpha^{\phi_{T}}(x))=0$, for every $x\in X$, thus $\mu$ is expansive for $\phi_{T}$.
\end{proof}

\begin{proof}[Proof of Theorem \ref{general3}]
There is a natural equivalence $\lambda: Y^{\tau,f}\to Y^{1,f}$ between $\phi^{\tau,f}$ and $\phi^{1,f}$ given by
$$
\lambda(x,t)=\left(x,\frac{t}{\tau(x)}\right).
$$
A direct computation shows that
$$
\lambda_*\mu^{\tau,f}=\mu^{1,f}.
$$
Therefore, Theorem \ref{general}-(A3) reduces the proof to the case $\tau=1$.

First assume that $T^{1,f}(\mu)$ is expansive for $\phi^{1,f}$. Then,
$T^{1,f}(\mu)$ is also expansive for $\phi^{1,f}_1$ by Theorem \ref{general2}.
Denoting by $Id$ the identity of $\left[0,\frac{1}{2}\right]$ one has
$$
\phi^{1,f}_1(x,s)=(x,s+1)=(f^{[s+1]}(x),s+1-[s+1])=(f(x),s)=(f\times Id)(x,s)
$$
for every $x\in X$ and $0\leq s\leq\frac{1}{2}$. Therefore, $T^{1,f}(\mu)$ is also expansive for
$f\times Id: X\times \left[0,\frac{1}{2}\right]\to X\times \left[0,\frac{1}{2}\right]$.
But clearly $T^{1,f}(\mu)$ restricted to $X\times \left[0,\frac{1}{2}\right]$ is
the product $\mu\times m$ where $m$ denotes the Lebesgue measure. So,
$\mu\times m$ is expansive for $f\times Id: X\times \left[0,\frac{1}{2}\right]\to X\times \left[0,\frac{1}{2}\right]$
and, then, $\mu$ is expansive for $f$.

Conversely, assume that $\mu$ is expansive for $f$.
Define the metric $d'$ in $X$ by $d'(x_1,x_2)=\min\{d(x_1,x_2), d(f(x_1),f(x_2))\}$, $\forall x_1,x_2\in X$.
Clearly $(X,d')$ is a compact metric space.
Since compact metrics in the same space are equivalent we have that $\mu$ is also expansive
for $f$ viewed as a homeomorphism of $(X, d')$.
Then, there is $0<\delta<\frac{1}{4}$ such that
$$
\mu(\hat{\Gamma}_\delta^{f}(x))=0, \quad\quad\forall x\in X,
$$
where $\hat{\Gamma}_\delta^{f}(x)$ above denotes the dynamical ball in Definition \ref{ar-mor} but with respect to the metric $d'$.

Now, fix $y=(x,t)\in Y^{1,f}$.

By Lemma \ref{suspension1} we have
$\Gamma_\delta(y)\subset \hat{\Gamma}_\delta^{f}(x)\times [0,1]$.
Then,
every Borel set $B\subset \Gamma_\delta(y)$ satisfies
$B\subset \hat{\Gamma}_\delta^{f}(x)\times [0,1]$
thus
$
\chi_B
\leq
\chi_{\hat{\Gamma}_\delta^{f}(x)\times [0,1]}
$
where $\chi$ above stands for characteristic function.
Taking $\tau=1$ in (\ref{eq9}) we obtain
$$
T^{1,f}(\mu)(B)=\int_X\int_0^1\chi_B(\phi_t^{1,f}(z)))dtd\mu(z)\leq
\int_X\int_0^1\chi_{\hat{\Gamma}_\delta^{f}(x)\times [0,1]}(\phi_t^{1,f}(z))dtd\mu(z)
=
$$
$$
\int_X\chi_{\hat{\Gamma}_\delta^{f}(x)}(z)d\mu(z)=\mu(\hat{\Gamma}_\delta^{f}(x))=0
$$
thus $T^{1,f}(\mu)(\Gamma_\delta(y))=0$.
Since $y\in Y^{1,f}$ is arbitrary we conclude that $T^{1,f}(\mu)$ is expansive for $\phi^{1,f}$ with expansivity constant $\delta$.
\end{proof}

\begin{proof}[Proof of Theorem \ref{chino}]
Suppose by contradiction that there is a measure-expansive flow $\phi$ of a closed surface with a non-trivial recurrence $x_0$.
Applying result by Gutierrez \cite{gu} and Theorem \ref{general}-(A3) we can assume that
$\phi$ is $C^1$.
It follows that the orbits of $\phi$ are all submanifolds of non-trivial
codimension. Since all such submanifolds have measure zero with respect to the
Lebesgue measure,
we obtain that $\phi$ exhibits at least one Borel measure which both has full support
and vanishes along the orbits.
It then follows from Theorem \ref{general}-(A1) that
$Sing(\phi)=\emptyset$.

Passing to a double covering if necessary we can assume that the surface is the two-dimensional torus $T^2$.

Since $x_0$ is not periodic there is a circle $X$ transversal to $\phi$ containing $x_0$
(see Proposition 5.1.2 in \cite{hh}).
Associated to this circle we have the return time $\tau: Dom(\tau)\subset X\to (0,\infty)$ and
the return map $f: Dom(f)\subset X\to X$, $f(x)=\phi_{\tau(x)}(x)$,
where $Dom(\cdot)$ stands for the domain operator.
Observe that $X$ cannot separates $T^2$ since $\phi$ has no singularities.

Since $x_0\in X$ is recurrent we have that $Dom(\tau)$ (and so $Dom(f)$) are not empty.
If there were $x_1\in Dom(\tau)$, then we can use the fact that
$X$ does not separate $T^2$ together with the Poincar\'e-Bendixon Theorem \cite{pdm} to prove that
the positive orbit of $x_1$ spirals toward a periodic orbit parallel to $X$.
In this case we would have $Dom(\tau)=\emptyset$ which is a contradiction, so, $Dom(\tau)$ (and hence $Dom(f)$) are the whole circle $X$.
From this we conclude that $X$ intersects every orbit of $\phi$ (see for instance
\cite{nz}, p. 146 in \cite{pdm} or Lemma 3 p. 106 in \cite{p}). Clearly
there is an equivalence $R$ between $\phi$ and the suspension flow
$\phi^{\tau,f}$ over $f$ with height function $\tau$.
Now take any non-atomic Borel measure $\mu$ of $X$ and its corresponding suspension
$T^{\tau,f}(\mu)$. In addition, $R^{-1}_*(T^{\tau,f}(\mu))$ vanishes along the orbits of $\phi$ and, since $\phi$ is measure-expansive, we conclude that
$R^{-1}_*(T^{\tau,f}(\mu))$ is expansive for $\phi$.
But $R$ is an equivalence between $\phi$ and $\phi^{\tau,f}$ so $T^{\tau,f}(\mu)=R_*(R^{-1}_*(T^{\tau,f}(\mu)))$
is expansive for $\phi^{\tau,f}$ by Theorem \ref{general}-(A3).
Applying Theorem \ref{general3} we conclude that $\mu$ is expansive for $f$.
We conclude that every non-atomic Borel measure of $X$ is expansive for $f$.
But this is absurd since there are no such homeomorphisms $f$
of the circle (c.f. \cite{mmm} or \cite{ms1}). This contradiction proves the result.
\end{proof}

\end{document}